\documentclass[11pt]{article}
\usepackage{amsmath,amsfonts,amssymb,amsthm}
\usepackage{bbm}
\usepackage{enumerate}
\usepackage{esint}
\usepackage[pagebackref,colorlinks,citecolor=blue,linkcolor=blue]{hyperref}
\usepackage[dvipsnames]{xcolor}
\usepackage{mathtools}  

\usepackage{comment}

\usepackage{geometry}
\numberwithin{equation}{section}
\theoremstyle{plain}
\newtheorem{theorem}[equation]{Theorem}
\newtheorem{corollary}[equation]{Corollary}
\newtheorem{lemma}[equation]{Lemma}
\newtheorem{proposition}[equation]{Proposition}

\theoremstyle{definition}

\newtheorem{example}[equation]{Example}
\newtheorem{remark}[equation]{Remark}

\numberwithin{equation}{section}

\newcommand{\R}{{\mathbb R}}
\newcommand{\N}{{\mathbb N}}

\newcommand{\Om}{\Omega}

\providecommand{\vint}[1]{\mathchoice
	{\mathop{\vrule width 5pt height 3 pt depth -2.5pt
			\kern -9pt \kern 1pt\intop}\nolimits_{\kern -5pt{#1}}}
	{\mathop{\vrule width 5pt height 3 pt depth -2.6pt
			\kern -6pt \intop}\nolimits_{\kern -3pt{#1}}}
	{\mathop{\vrule width 5pt height 3 pt depth -2.6pt
			\kern -6pt \intop}\nolimits_{\kern -3pt{#1}}}
	{\mathop{\vrule width 5pt height 3 pt depth -2.6pt
			\kern -6pt \intop}\nolimits_{\kern -3pt{#1}}}}

\newcommand{\eps}{\varepsilon}
\newcommand{\loc}{\mathrm{loc}}

\newcommand{\BV}{\mathrm{BV}}

\newcommand{\ch}{\text{\raise 1.3pt \hbox{$\chi$}\kern-0.2pt}}

\newcommand{\mres}{\mathbin{\vrule height 2ex depth 2.2pt width
		0.12ex\vrule height -0.3ex depth 2.2pt width .5ex}}

\DeclareMathOperator{\dive}{div}

\DeclareMathOperator{\Var}{Var}

\begin{document}
	\title{A sharp lower bound for a class of non-local approximations of the total variation
		\footnote{
			\hfill \break {\it Keywords\,}: Function of bounded variation,
			nonlocal functional, Cantor part
	}}
	\author{Panu Lahti}
	
	\maketitle
	
	\begin{abstract}
		We study a class of non-local functionals that was introduced by 
		Brezis--Seeger--Van Schaftingen--Yung \cite{BSSY2},
		and can be used to characterize functions of bounded variation.
		We give a new lower bound for the liminf of these functionals,
		involving the three different parts of the total variation, with sharp coefficients.
	\end{abstract}
	
	\section{Introduction}
	
	In the past two decades,
	there has been widespread interest in characterizing Sobolev and BV (bounded variation)
	functions by means of non-local functionals;
	see e.g. \cite{ABBF,AGMP,BBM,Bre,BN18,BN16,Dav,FMS}.
	One of these non-local functionals is defined as follows.
	Let $n\in \N$ and let $\gamma\in\R$.
	Define the measure
	\[
	\nu_{\gamma}(A):=\iint_A |x-y|^{\gamma-n}\,dy\,dx,
		\quad A\subset \R^n\times \R^n.
	\]
	Let $\Om\subset \R^n$ be open. For a measurable function $f\colon \Om\to \R$,
	let
	\[
	E_{\gamma,\lambda}(f,\Om)
	:=\left\{(x,y)\in \Om\times \Om\colon
	|f(x)-f(y)|>\lambda|x-y|^{1+\gamma}\right\}, \quad \lambda \in (0,\infty),
	\]
	and then define the functional
		\begin{align*}
		F_{\gamma,\lambda}\left(f,\Om\right)
		:=\lambda\nu_\gamma (E_{\gamma,\lambda}(f,\Om)).
	\end{align*}
	This was studied in the case $\gamma=n$ in \cite{BVSY,Pol}, and then
	more generally for $\gamma\in\R$ in
	\cite{BSSY,BSSY2}.
	
	Denoting $x=(x_1,\ldots,x_n)\in \R^n$, let
	\[
	C_n:=\int_{\mathbb S^{n-1}}|x_1|\,d\mathcal H^{n-1}(x),
	\]
	where $\mathcal H^{n-1}$ is the $(n-1)$-dimensional Hausdorff measure,
	and $\mathbb S^{n-1}$ is the unit sphere.
	In response to a question posed 
	by Brezis--Seeger--Van Schaftingen--Yung  \cite{BSSY},
	very recently Picenni \cite{Pic} showed that if $f\in L^1_{\loc}(\R^n)$ with $\Var(f,\R^n)<\infty$,
	and $\gamma>0$, then
	\[
	\liminf_{\lambda \to\infty}
	F_{\gamma,\lambda}(f,\R^n)
	\ge \frac{C_n}{\gamma}|D^af|(\R^n)
	+\frac{C_n\gamma }{2(1+2\gamma)(1+\gamma)}|D^cf|(\R^n)
	+\frac{C_n}{\gamma+1}|D^jf|(\R^n),
	\]
	where $|D^af|$, $|D^cf|$, and $|D^jf|$ denote the absolutely continuous, Cantor, and jump parts
	of the variation measure.
	The proof in \cite{Pic} is based on first considering the case $n=1$, and
	analyzing the blow-up behavior of $f$ at $|Df|$-a.e. point in $\R$;
	for $|D^cf|$-a.e. point,
	this blow-up behavior has a less neat representation than for
	$|D^a f|$-a.e. point and $|D^j f|$-a.e. point,
	and this leads to the somewhat small coefficient in front of $|D^cf|(\R^n)$.
	
	We analyze the contribution of $|D^cf|(\R)$
	using a different, more ``global'' strategy.
	Then we generalize to higher dimensions in the usual way.
	The following is our main theorem.
	\begin{theorem}\label{thm:main}
	Let $\gamma>0$, let $\Om\subset \R^n$ be open, and let $f\in L^1_{\loc}(\Om)$ with
	$\Var(f,\Om)<\infty$. Then
	\[
		\liminf_{\lambda \to\infty}
		F_{\gamma,\lambda}(f,\Om)
		\ge \frac{C_n}{\gamma}|D^af|(\Om)
		+\frac{C_n}{\gamma+1}|D^cf|(\Om)
		+\frac{C_n}{\gamma+1}|D^jf|(\Om).
	\]
	\end{theorem}
	
	Picenni \cite{Pic} also shows that for $f\in L^1_{\loc}(\R^n)$ with
	$\Var(f,\R^n)<\infty$ and $|D^c f|(\R^n)=0$, we have
	\begin{equation}\label{eq:SBV limit}
	\lim_{\lambda \to\infty}
	F_{\gamma,\lambda}(f,\R^n)
	= \frac{C_n}{\gamma}|D^af|(\R^n)
	+\frac{C_n}{\gamma+1}|D^jf|(\R^n).
	\end{equation}
	This implies that the coefficients  of $|D^af|(\Om)$ and $|D^jf|(\Om)$ 
	in Theorem \ref{thm:main} are sharp.
	In Example \ref{ex:sharpness} we show that also the coefficient of $|D^cf|(\Om)$ is sharp,
	and that the limit $\lim_{\lambda \to\infty}
	F_{\gamma,\lambda}(f,\Om)$ may fail to exist when $|D^c f|(\Om)>0$.
	
	\section{Preliminaries}\label{sec:definitions}
	
	We work in the Euclidean space $\R^n$ with $n\ge 1$.
	We denote the $n$-dimensional Lebesgue outer measure by $\mathcal L^n$.
	We denote the characteristic function of a set $A\subset\R^n$ by $\mathbbm{1}_A\colon \R^n\to \{0,1\}$.
	
	The theory of $\BV$ functions that we rely on can be found in the monograph
	Ambrosio--Fusco--Pallara \cite{AFP}.
	By $\Om$ we always denote an open subset of $\R^n$.
	A function
	$f\in L^1(\Omega)$ is of bounded variation,
	denoted $f\in \BV(\Omega)$, if its weak derivative
	is an $\R^{n}$-valued Radon measure with finite total variation. This means that
	there exists a (unique) Radon measure $Df$
	such that for all $\varphi\in C_c^1(\Omega)$, the integration-by-parts formula
	\[
	\int_{\Omega}f\frac{\partial\varphi}{\partial x_k}\,d\mathcal L^n
	=-\int_{\Omega}\varphi\,d(Df)_k,\quad k=1,\ldots,n,
	\]
	holds.
	The total variation of $Df$ is denoted by $|Df|$.
	If we do not know a priori that a mapping $f\in L^1_{\loc}(\Om)$
	is a BV function, we consider
	\[
	\Var(f,\Om):=\sup\left\{\int_{\Om}f\dive\varphi\,d\mathcal L^n,\,\varphi\in C_c^{1}(\Om;\R^{n}),
	\,|\varphi|\le 1\right\}.
	\]
	If $\Var(f,\Om)<\infty$, then the $\R^{n}$-valued Radon measure $Df$
	exists and $\Var(f,\Om)=|Df|(\Om)$
	by the Riesz representation theorem, and $f\in\BV(\Om)$ provided that $f\in L^1(\Om)$.
	
	Denote by $S_f\subset \Om$ the set of non-Lebesgue points of $f\in \BV_{\loc}(\Om)$.
	We write the Radon-Nikodym decomposition of the variation measure of $f$ into the absolutely continuous and singular parts with respect to $\mathcal L^n$
	as
	\[
	Df=D^a f+D^s f.
	\]
	Furthermore, we define the Cantor and jump parts of $Df$ as
	\[
	D^c f\coloneqq  D^s f\mres (\Om\setminus S_f),\qquad D^j f\coloneqq D^s f\mres S_f.
	\]
	Here
	\[
	D^s f \mres S_f(A):=D^s f (S_f\cap A),\quad \textrm{for } D^s f\textrm{-measurable } A\subset \Om.
	\]
	Thus we get the decomposition
	\[
	Df=D^a f+ D^c f+ D^j f.
	\]

	\section{Proof of Theorem \ref{thm:main}}
	
	We start with the following simple lemma.
	
	\begin{lemma}\label{lem:uniform convergence}
		Let $\gamma,\lambda>0$. Let $\Om\subset\R$ be open and let $f\colon \Om\to \R$ 
		be a measurable function such that
		there is a bounded set $W\subset \Om$ and
		$\lambda_0<\lambda$ for which we have
		\begin{equation}\label{eq:contained in W W}
		E_{\gamma,\lambda_0}(f',\Om)\subset W\times W.
		\end{equation}
	for all $f'$ with $\Vert f'-f\Vert_{L^{\infty}(\Om)}$ sufficiently small.
	Moreover, assume that
	\begin{equation}\label{eq:lambda point assumption}
	\nu_{\gamma}(\left\{(x,y)\in \Om\times \Om\colon
	|f(x)-f(y)|=\lambda|x-y|^{1+\gamma}\right\})=0.
	\end{equation}
	Then
	\[
	\lim_{\delta\to 0}
	\sup_{f^*\colon \Vert f^*-f\Vert_{L^{\infty}(\Om)}\le\delta}
	\left|F_{\gamma,\lambda}(f^*,\Om)-F_{\gamma,\lambda}(f,\Om)\right| =0,
	\]
	where the supremum is taken over measurable functions $f^*$ on $\Om$ for which
	$\Vert f^*-f\Vert_{L^{\infty}(\Om)}\le\delta$.
	\end{lemma}
	\begin{proof}
		Assume by contradiction that there exists a sequence of measurable functions $f^i$
		on $\Om$ such that $\Vert f^i-f\Vert_{L^{\infty}(\Om)}\to 0$ as $i\to\infty$, but
		$\left|F_{\gamma,\lambda}(f^i,\Om)-F_{\gamma,\lambda}(f,\Om)\right|>\eps>0$ for all
		$i\in\N$.
		 
		Let $\lambda'<\lambda$.
		Then
		$E_{\gamma,\lambda}(f^i,\Om)\setminus E_{\gamma,\lambda'}(f,\Om)\to\emptyset$ as $i\to\infty$,
		and so by choosing $\lambda'$ to be in the interval $(\lambda_0,\lambda)$,
		using \eqref{eq:contained in W W} we have
		\[
		\lim_{i\to\infty}\nu_\gamma(E_{\gamma,\lambda}(f^i,\Om)\setminus E_{\gamma,\lambda'}(f,\Om))=0.
		\]
		Then
		\[
		\lambda\nu_\gamma(E_{\gamma,\lambda'}(f,\Om))
			\ge\limsup_{i\to\infty}\lambda\nu_\gamma(E_{\gamma,\lambda}(f^i,\Om))
			=\limsup_{i\to\infty}F_{\gamma,\lambda}(f^i,\Om).
		\]
		Using the assumption
		\eqref{eq:lambda point assumption}, we get
		\[
		\lim_{\lambda'\nearrow \lambda}\nu_\gamma(E_{\gamma,\lambda'}(f,\Om))=\nu_\gamma(E_{\gamma,\lambda}(f,\Om)).
		\]
		Combining, we get
		\[
		F_{\gamma,\lambda}(f,\Om)
		\ge \limsup_{i\to\infty}F_{\gamma,\lambda}(f^i,\Om).
		\]
		
		Then let $\lambda'>\lambda$.
		Now
		$E_{\gamma,\lambda'}(f,\Om)\setminus E_{\gamma,\lambda}(f^i,\Om)\to\emptyset$ as $i\to\infty$,
		and so by \eqref{eq:contained in W W}, we have
		\[
		\lim_{i\to\infty}\nu_\gamma(E_{\gamma,\lambda'}(f,\Om)\setminus E_{\gamma,\lambda}(f^i,\Om))=0.
		\]
		Then
		\[
		\lambda\nu_\gamma(E_{\gamma,\lambda'}(f,\Om))
		\le\liminf_{i\to\infty}\lambda\nu_\gamma(E_{\gamma,\lambda}(f^i,\Om))
		=\liminf_{i\to\infty}F_{\gamma,\lambda}(f^i,\Om).
		\]
		We also have
		\[
		\lim_{\lambda'\searrow \lambda}\nu_\gamma(E_{\gamma,\lambda'}(f,\Om))=\nu_\gamma(E_{\gamma,\lambda}(f,\Om)).
		\]
		Combining, we get
		\[
		F_{\gamma,\lambda}(f,\Om)
		\le\liminf_{i\to\infty}F_{\gamma,\lambda}(f^i,\Om),
		\]
		giving a contradiction and proving the result.
	\end{proof}
	
	The following theorem is the crux of our argument.
	
	\begin{theorem}\label{thm:1d case preli}
		Let $-\infty<b_1\le b_2<\infty$ and $-\infty<L_1\le L_2<\infty$, and let 
		$f$ be increasing with $f(x)=b_1$ for $x\le L_1$ and $f(x)=b_2$ for $x\ge L_2$.
		Let $\gamma,\lambda>0$.
		Then
		\[
		F_{\gamma,\lambda}(f,\R)\ge \frac{2(b_2-b_1)}{\gamma+1}.
		\]
	\end{theorem}
	\begin{proof}
	If $b_1=b_2$ then both sides are zero, and so we can assume that
	$b_1<b_2$. By scaling, we can assume that $b_1=0$ and $b_2=1$.
		
	Observe that
	\begin{equation}\label{eq:finiteness}
			F_{\gamma,\lambda}(f,\R)
			=\lambda\nu_\gamma (E_{\gamma,\lambda}(f,\R))
			\le \lambda \nu_\gamma
			([L_1-\lambda^{-1/(1+\gamma)},\,L_2+\lambda^{-1/(1+\gamma)}]^2)<\infty.
	\end{equation}
		
	First suppose that for some $M\in\N$ and for some numbers $-\infty<a_1<\ldots<a_M<\infty $, we have
	\begin{equation}\label{eq:fM form}
	f(x)= \frac{1}{M}\sum_{j=1}^M \mathbbm{1}_{[a_j,\infty)}=:\sum_{j=1}^M f_j.
	\end{equation}
	We compute
	\begin{align*}
	&\lambda \nu_\gamma\left(\left\{(x,y)\in \R\times \R\colon
	\frac{|f_1(x)-f_1(y)|}{|x-y|^{1+\gamma}}>\lambda \right\}\right)\\
	&\quad  = 2\lambda \int_{a_1- (M\lambda )^{-1/(1+\gamma)}}^{a_1}\int_{a_1}^{(M\lambda )^{-1/(1+\gamma)}+x}(y-x)^{\gamma-1}\,dy\,dx\\
	&\quad  = \frac{2\lambda }{\gamma}\int_{a_1- (M\lambda )^{-1/(1+\gamma)}}^{a_1} \left[(M\lambda )^{-\gamma/(1+\gamma)}-(a_1-x)^{\gamma}
	\right]\,dx\\
	&\quad  = \frac{2\lambda }{\gamma}\left[(M\lambda )^{-1}-\frac{1}{\gamma+1}(M\lambda )^{-1}\right]\\
	& \quad = \frac{2}{\gamma+1}\frac{1}{M}.
	\end{align*}
This of course holds generally:
\begin{equation}\label{eq:one step function}
	\lambda \nu_\gamma\left(\left\{(x,y)\in \R\times \R\colon
	\frac{|f_j(x)-f_j(y)|}{|x-y|^{1+\gamma}}>\lambda \right\}\right)
	= \frac{2}{\gamma+1}\frac{1}{M}
\end{equation}
for all $j=1,\ldots,M$.
	Fix $k\in \{1,\ldots, M-1\}$. We consider the two quantities
	\[
	A_1:=F_{\gamma,\lambda}\left(\sum_{j=1}^{k+1} f_j,\R\right)
	-F_{\gamma,\lambda}\left(\sum_{j=1}^{k} f_j,\R\right)
	\]
	and
	\[
	A_2:=F_{\gamma,\lambda}\left(\sum_{j=1}^{k+1} f_j+M^{-1}\mathbbm{1}_{(-\infty,a_1)},\R\right)
	-F_{\gamma,\lambda}\left(\sum_{j=1}^{k} f_j+M^{-1}\mathbbm{1}_{(-\infty,a_1)},\R\right).
	\]
	Note that all of the terms are finite, as can be seen just as in \eqref{eq:finiteness}.
	We have
	\begin{align*}
	&A_1-A_2\\
	&\quad =\lambda \nu_\gamma\left(\left\{(x,y)\in \R\times \R\colon
\frac{|\sum_{j=1}^{k+1} f_j(x)-\sum_{j=1}^{k+1} f_j(y)|}{|x-y|^{1+\gamma}}>\lambda \right\}\right)\\
&\qquad -\lambda \nu_\gamma\left(\left\{(x,y)\in \R\times \R\colon
\frac{|\sum_{j=1}^{k} f_j(x)-\sum_{j=1}^{k} f_j(y)|}{|x-y|^{1+\gamma}}>\lambda \right\}\right)\\
&\qquad -\lambda \nu_\gamma\left(\left\{(x,y)\in \R\times \R\colon
\frac{|\sum_{j=2}^{k+1} f_j(x)-\sum_{j=2}^{k+1} f_j(y)|}{|x-y|^{1+\gamma}}>\lambda \right\}\right)\\
&\qquad +\lambda \nu_\gamma\left(\left\{(x,y)\in \R\times \R\colon
\frac{|\sum_{j=2}^{k} f_j(x)-\sum_{j=2}^{k} f_j(y)|}{|x-y|^{1+\gamma}}>\lambda \right\}\right)\\
	&\quad =2\lambda \nu_\gamma\left(\left\{x\le a_1,\,y\ge a_{k+1}\colon
	\frac{|\sum_{j=1}^{k+1} f_j(x)-\sum_{j=1}^{k+1} f_j(y)|}{|x-y|^{1+\gamma}}>\lambda \right\}\right)\\
	&\qquad -2\lambda \nu_\gamma\left(\left\{x\le a_1,\,y\ge a_{k+1}\colon
	\frac{|\sum_{j=1}^{k} f_j(x)-\sum_{j=1}^{k} f_j(y)|}{|x-y|^{1+\gamma}}>\lambda \right\}\right)\\
	&\qquad -2\lambda \nu_\gamma\left(\left\{x\le a_1,\,y\ge a_{k+1}\colon
	\frac{|\sum_{j=2}^{k+1} f_j(x)-\sum_{j=2}^{k+1} f_j(y)|}{|x-y|^{1+\gamma}}>\lambda \right\}\right)\\
	&\qquad +2\lambda \nu_\gamma\left(\left\{x\le a_1,\,y\ge a_{k+1}\colon
	\frac{|\sum_{j=2}^{k} f_j(x)-\sum_{j=2}^{k} f_j(y)|}{|x-y|^{1+\gamma}}>\lambda \right\}\right),
	\end{align*}
because all of the other ``interactions'' between $x$ and $y$ cancel out.
Thus
\begin{equation}\label{eq:A1 A2}
	\begin{split}
			&A_1-A_2\\
	&\quad =2\lambda \nu_\gamma\left(\left\{x\le a_1,\,y\ge a_{k+1}\colon
	\frac{(k+1)/M}{|x-y|^{1+\gamma}}>\lambda \right\}\right)\\
	&\qquad -4\lambda \nu_\gamma\left(\left\{x\le a_1,\,y\ge a_{k+1}\colon
	\frac{k/M}{|x-y|^{1+\gamma}}>\lambda \right\}\right)\\
	&\qquad +2\lambda \nu_\gamma\left(\left\{x\le a_1,\,y\ge a_{k+1}\colon
	\frac{(k-1)/M}{|x-y|^{1+\gamma}}>\lambda \right\}\right).
	\end{split}
\end{equation}
	If $l\in\N$ and $(l/(M\lambda ))^{1/(1+\gamma)}\ge a_{k+1}-a_1$, then we can calculate
	\begin{align*}
		&\nu_\gamma\left(\left\{x\le a_1,\,y\ge a_{k+1}\colon
		\frac{l/M}{|x-y|^{1+\gamma}}>\lambda \right\}\right)\\
		&\quad =\int_{  a_{k+1}-(l/(M\lambda ))^{1/(1+\gamma)}  }^{a_1}
		\int_{a_{k+1}}^{(l/(M\lambda ))^{1/(1+\gamma)}+x}
		(y-x)^{\gamma-1}\,dy\,dx\\
		&\quad = \frac{1}{\gamma}\int_{  a_{k+1}-(l/(M\lambda ))^{1/(1+\gamma)}  }^{a_1}
		\left[(l/(M\lambda ))^{\gamma/(1+\gamma)}-(a_{k+1}-x)^{\gamma}\right]\,dx\\
		&\quad = \frac {1}{\gamma} (a_1-a_{k+1}+(l/(M\lambda ))^{1/(1+\gamma)})
		(l/(M\lambda ))^{\gamma/(1+\gamma)}\\
		&\qquad +\frac {1}{\gamma(\gamma+1)}\left[
		(a_{k+1}-a_1)^{\gamma+1}-l/(M\lambda )\right]\\
		&\quad = -\frac{1}{\gamma}(a_{k+1}-a_1)(l/(M\lambda ))^{\gamma/(1+\gamma)}
		+\frac{l/(M\lambda) }{\gamma+1}+\frac {1}{\gamma(\gamma+1)}(a_{k+1}-a_1)^{\gamma+1}.
	\end{align*}
	We record this fact: if $(l/(M\lambda ))^{1/(1+\gamma)}\ge a_{k+1}-a_1$, then
	\begin{equation}\label{eq:general k}
		\begin{split}
			&\nu_\gamma\left(\left\{x\le a_1,\,y\ge a_{k+1}\colon
			\frac{l/M}{|x-y|^{1+\gamma}}>\lambda \right\}\right)\\
			&= \frac{l/(M\lambda) }{\gamma+1}
			-\frac{1}{\gamma}(a_{k+1}-a_1)(l/(M\lambda ))^{\gamma/(1+\gamma)}
			+\frac {1}{\gamma(\gamma+1)}(a_{k+1}-a_1)^{\gamma+1}.
		\end{split}
	\end{equation}

	Now we estimate $A_1-A_2$; we consider three cases according to how far the points $a_1$ and $a_{k+1}$ are from each other.\\
	
	\textbf{Case 1.}\\
	First consider the case $((k-1)/(M\lambda ))^{1/(1+\gamma)}\ge a_{k+1}-a_1$.
	Now \eqref{eq:general k} is applicable with $l=k-1,k,k+1$.
	By \eqref{eq:A1 A2},
	\begin{align*}
	 &A_1-A_2\\
	 &=	\frac{2\lambda }{\gamma}(a_{k+1}-a_1)\left[
	 -((k+1)/(M\lambda ))^{\gamma/(1+\gamma)}
	 +2(k/(M\lambda ))^{\gamma/(1+\gamma)}
	 -((k-1)/(M\lambda ))^{\gamma/(1+\gamma)}\right]\\
	&\ge 0
\end{align*}
by concavity of the function $t\mapsto t^{\gamma/(1+\gamma)}$.\\

	\textbf{Case 2.}\\
	Then consider the case $(k/(M\lambda ))^{1/(1+\gamma)}\le a_{k+1}-a_1$.
	From \eqref{eq:A1 A2},
	\[
	A_1-A_2	 =2\lambda \nu_\gamma\left(\left\{x\le a_1,\,y\ge a_{k+1}\colon
	\frac{(k+1)/M}{|x-y|^{1+\gamma}}>\lambda \right\}\right)
	\ge 0.\\
	\]
	
	\textbf{Case 3.}\\
	Then consider the case $((k-1)/(M\lambda ))^{1/(1+\gamma)}\le a_{k+1}-a_1\le (k/(M\lambda ))^{1/(1+\gamma)}$.
	From \eqref{eq:A1 A2} and \eqref{eq:general k},
	\begin{align*}
	&\frac{A_1-A_2}{2\lambda}\\
	&\quad	 = \nu_\gamma\left(\left\{x\le a_1,\,y\ge a_{k+1}\colon
	\frac{(k+1)/M}{|x-y|^{1+\gamma}}>\lambda \right\}\right)\\
	&\qquad -2 \nu_\gamma\left(\left\{x\le a_1,\,y\ge a_{k+1}\colon
	\frac{k/M}{|x-y|^{1+\gamma}}>\lambda \right\}\right)\\
	&\quad=\frac{(1-k)/(M\lambda )}{\gamma+1}
	 +\frac{1}{\gamma}(a_{k+1}-a_1)\left[ 2(k/(M\lambda ))^{\gamma/(1+\gamma)}-((k+1)/(M\lambda ))^{\gamma/(1+\gamma)}\right]
	 \\
	&\qquad  -\frac{1}{\gamma(1+\gamma)}(a_{k+1}-a_1)^{1+\gamma}.
	\end{align*}
Note that this is a concave function of the quantity
$(a_{k+1}-a_1)$.
Thus it is at its smallest at one of the end points 
$a_{k+1}-a_1=((k-1)/(M\lambda ))^{1/(1+\gamma)}$, $a_{k+1}-a_1= (k/(M\lambda ))^{1/(1+\gamma)}$.
But these cases were already handled in Cases 1 and 2. Thus again
\[
A_1-A_2\ge 0.
\]
In total, we always have $A_1\ge A_2$.
Using this inductively, we get
\begin{align*}
&F_{\gamma,\lambda}\left(\sum_{j=1}^{k+1} f_j,\R\right)
-F_{\gamma,\lambda}\left(\sum_{j=1}^{k} f_j,\R\right)\\
&\ge F_{\gamma,\lambda}\left(\sum_{j=1}^{k+1} f_j+M^{-1}\mathbbm{1}_{(-\infty,a_1)},\R\right)
-F_{\gamma,\lambda}\left(\sum_{j=1}^{k} f_j+M^{-1}\mathbbm{1}_{(-\infty,a_1)},\R\right)\\
&\ge \ldots\\
&\ge F_{\gamma,\lambda}\left(\sum_{j=1}^{k+1} f_j+M^{-1}\sum_{j=1}^{k}\mathbbm{1}_{(-\infty,a_j)},\R\right)
-F_{\gamma,\lambda}\left(\sum_{j=1}^{k} f_j+M^{-1}\sum_{j=1}^{k}\mathbbm{1}_{(-\infty,a_j)},\R\right)\\
&= F_{\gamma,\lambda}\left(f_{k+1}+k/M,\R\right)
-F_{\gamma,\lambda}\left(k/M,\R\right)\\
&= F_{\gamma,\lambda}\left(f_{k+1},\R\right)\\
&=\frac{2}{\gamma+1}\frac{1}{M}
\end{align*}
by \eqref{eq:one step function}.
In total
\begin{align*}
F_{\gamma,\lambda}\left(\sum_{j=1}^{M} f_j,\R\right)
\ge F_{\gamma,\lambda}\left( f_1,\R\right)+\frac{2}{\gamma+1}\frac{M-1}{M}
= \frac{2}{\gamma+1}
\end{align*}
by \eqref{eq:one step function}.

Finally, we consider a general $f$; in other words,
$f$ is increasing with $f(x)=0$ for $x\le L_1$ and $f(x)=1$ for $x\ge L_2$.
We find a sequence of step functions $\{f^i\}_{i=1}^{\infty}$ of the form
\eqref{eq:fM form},
with $\Vert f^i- f\Vert_{L^{\infty}(\R)}\to 0$,
 and  $f^i(x)=0$ for $x< L_1$ and $f^i(x)=1$ for $x> L_2$.
Let $\lambda_0\in (0,\lambda)$.
Note that for all $i\in\N$ and all $x,y\in\R$,
\begin{align*}
&\left\{(x',y')\in \R\times \R\colon
	\frac{|f^i(x')-f^i(y')|}{|x'-y'|^{1+\gamma}}> \lambda_0\right\}\\
&\quad \subset [L_1-\lambda_0^{-1/(1+\gamma)},\,L_2+\lambda_0^{-1/(1+\gamma)}]
\times [L_1-\lambda_0^{-1/(1+\gamma)},\,L_2+\lambda_0^{-1/(1+\gamma)}],
\end{align*}
and the same holds with $f^i$ replaced by $f$.
Now by Lemma \ref{lem:uniform convergence}, we have that
\[
F_{\gamma,\lambda'}(f,\R)
=\lim_{i\to\infty}F_{\gamma,\lambda'}(f^i,\R)
\ge \frac{2}{\gamma+1}
\]
for all $\lambda'>0$ such that
$\nu_{\gamma}(	\left\{(x,y)\in \Om\times \Om\colon
|f(x)-f(y)|=\lambda'|x-y|^{1+\gamma}\right\})=0$;
note that this is true for all except at most countably many $\lambda'$, by the fact that
$\nu_{\gamma}$ is a locally finite measure.
Thus choosing $\lambda'\searrow \lambda$, we get
\[
F_{\gamma,\lambda}\left(f,\R\right)\ge \frac{2}{\gamma+1}.
\]
\end{proof}
	
		\begin{corollary}\label{cor:1d case preli}
		Let $-\infty<b_1\le b_2<\infty$ and $-\infty<L_1\le L_2<\infty$, and let 
		$f$ be increasing with $f(x)=b_1$ for $x\le L_1$ and $f(x)=b_2$ for $x\ge L_2$.
		Let $\gamma>0$.
		Then for every $\delta>0$, we have
		\[
		\liminf_{\lambda\to\infty}\lambda \nu_\gamma(\left\{(x,y)\in (L_1-\delta,L_2+\delta)^2\colon
		|f(x)-f(y)|>\lambda|x-y|^{1+\gamma}\right\})
		\ge \frac{2(b_2-b_1)}{\gamma+1}.
		\]
	\end{corollary}
\begin{proof}
	Note that
	\[
	\left\{(x,y)\in (L_1-\delta,L_2+\delta)^2\colon
	|f(x)-f(y)|>\lambda|x-y|^{1+\gamma}\right\}
	=E_{\gamma,\lambda}(f,\R)
	\]
	for sufficiently large $\lambda$.
	Now the conclusion follows from Theorem \ref{thm:1d case preli}.
\end{proof}
	
	Next, we can estimate the relation between the functional and the total variation.
	By \cite[Theorem 1.4]{BSSY} we know that for an open interval $\Om\subset \R$ and $f\in L^1_{\loc}(\Om)$,
	 we have
	\begin{equation}\label{eq:sup bound}
	F_{\gamma,\lambda}\left(f,\Om\right)
	\le C'\Var(f,\Om)
	\end{equation}
	for some constant $C'$ depending only on $n$ and $\gamma$.
	
	\begin{proposition}\label{thm:1d case variation}
	Let $\gamma>0$, let $\Om\subset \R$ be open and let 
	$f\in L^1_{\loc}(\Om)$ with $\Var(f,\Om)<\infty$.
	Then
	\[
	\liminf_{\lambda \to \infty}F_{\gamma,\lambda}\left(f,\Om\right)
	\ge \frac{2}{\gamma+1}|Df|(\Om).
	\]
\end{proposition}
\begin{proof}
	First assume that $\Om$ is connected, that is, $\Om=(a,b)$ for some $-\infty\le a<b\le \infty$.
	We can assume that $f$ is the pointwise representative
	\begin{equation}\label{eq:pw repr}
	f(x)=\limsup_{r\to 0}\vint{B(x,r)}f\,d\mathcal L^1
	\quad\textrm{for all }x\in \Om.
	\end{equation}
	By BV theory on the real line, see \cite[Section 3.2]{AFP},
	we know that $\lim_{x\to a^+}f(x)$ exists, and that
	\[
	f(y)=\lim_{x\to a^+}f(x)+|Df|((a,y])\quad \textrm{for all Lebesgue points }y\in \Om.
	\]
	Fix $0<\eps<1$. For $|Df|$-a.e. $x\in \Om$, for arbitrarily small $r>0$ we have that
	\[
	\frac{|Df([x-r,x+r])|}{|Df|([x-r,x+r])}>1-\eps^2.
	\]
Thus by Vitali's covering theorem,
we find a sequence of pairwise disjoint closed intervals $[c_j,d_j]\subset \Om$, $j\in\N$, such that
\begin{equation}\label{eq:cj and dj}
\frac{|Df([c_j,d_j])|}{|Df|([c_j,d_j])}>1-\eps^2,
\end{equation}
and $|Df|\left(\Om\setminus \bigcup_{j=1}^{\infty}[c_j,d_j]\right)=0$.
For sufficiently large $N\in\N$, we have
\begin{equation}\label{eq:j one to N}
|Df|\left(\Om\setminus \bigcup_{j=1}^{N}[c_j,d_j]\right)<\eps^2.
\end{equation}
We find $\delta>0$ such that the intervals $[c_j-\delta,d_j+\delta]$, $j=1,\ldots,N$,
are still pairwise disjoint
and contained in $\Om$.
Define a measure $\mu$ on $\Om$ as follows:
\[
\begin{cases}
\mu:=(Df)_+\quad\textrm{on }[c_j,d_j]\quad\textrm{if }Df([c_j,d_j])>0\\
\mu:=(Df)_-\quad\textrm{on }[c_j,d_j]\quad\textrm{if }Df([c_j,d_j])<0,
\end{cases}
\]
and $\mu:=0$ outside these intervals.
From \eqref{eq:cj and dj}, we have
\begin{equation}\label{eq:epsilon condition}
|Df-\mu|([c_j,d_j])<\eps^2 |Df|([c_j,d_j]).
\end{equation}
Define
\[
g(y)=\lim_{x\to a^+}f(x)+\mu((a,y])\quad \textrm{for }y\in \Om.
\]
Then
\begin{equation}\label{eq:D fg difference}
	\begin{split}
|D(f-g)|(\Om)
&=|Df-\mu|(\Om)\\
&= \sum_{j=1}^N \big[|Df-\mu|([c_j,d_j])+|Df|\left(\Om\setminus \bigcup_{j=1}^N [c_j,d_j]\right)\\
&\le \eps^2|Df|(\Om)+\eps^2\quad\textrm{by }\eqref{eq:j one to N},\,\eqref{eq:epsilon condition}.
	\end{split}
\end{equation}
Note that
\begin{equation}\label{eq:pV estimate}
	\begin{split}
	\sum_{j=1}^N|g(c_j-\delta/2)-g(d_j+\delta/2)|
	&=\sum_{j=1}^N|Dg([c_j,d_j])|\\
	&= \sum_{j=1}^N|Dg|([c_j,d_j])\\
	&\ge \sum_{j=1}^N|Df|([c_j,d_j])-\eps^2(1+|Df|(\Om))\quad\textrm{by }\eqref{eq:D fg difference}\\
	&\ge (|Df|(\Om)-\eps^2)-\eps^2(1+|Df|(\Om))\quad\textrm{by }\eqref{eq:j one to N}.
	\end{split}
\end{equation}
Note also that
\begin{align*}
&\left\{(x,y)\in \Om\times \Om\colon
\frac{|f(x)-f(y)|}{|x-y|^{1+\gamma}}>(1+\eps)\lambda\right\}\\
&\quad\subset  \left\{(x,y)\in \Om\times \Om\colon
\frac{|g(x)-g(y)|}{|x-y|^{1+\gamma}}>\lambda\right\}\\
&\qquad \cup  \left\{(x,y)\in \Om\times \Om\colon
\frac{|(f-g)(x)-(f-g)(y)|}{|x-y|^{1+\gamma}}
>\eps\lambda\right\}.
\end{align*}
Thus we can estimate
\begin{align*}
&\liminf_{\lambda \to \infty}F_{\gamma,\lambda}\left(f,\Om\right)\\
&\qquad \ge \frac{1}{1+\eps}\liminf_{\lambda \to \infty}F_{\gamma,(1+\eps)\lambda}\left(g,\Om\right)
-\eps^{-1}\limsup_{\lambda \to \infty}F_{\gamma,\eps\lambda}\left(g-f,\Om\right)\\
&\qquad\ge \frac{1}{1+\eps}\liminf_{\lambda \to \infty}F_{\gamma,(1+\eps)\lambda}\left(g,\Om\right)
-C'\eps^{-1}|D(f-g)|(\Om)\quad\textrm{by }\eqref{eq:sup bound}\\
&\qquad\ge \frac{1}{1+\eps}\sum_{j=1}^N\liminf_{\lambda \to \infty}F_{\gamma,(1+\eps)\lambda}\left(g,(c_j-\delta,d_j+\delta)\right)
-C'\eps^{-1}\eps^2 (1+|Df|(\Om))\quad\textrm{by }\eqref{eq:D fg difference}\\
&\qquad\ge \frac{1}{1+\eps}\sum_{j=1}^N\frac{2}{\gamma+1}|g(c_j-\delta/2)-g(d_j+\delta/2)|-C'\eps |Df|(\Om)
\quad\textrm{by Corollary }\ref{cor:1d case preli}\\
&\qquad\ge \frac{1}{1+\eps}\frac{2}{\gamma+1}(|Df|(\Om)-\eps^2)-\eps^2(1+|Df|(\Om))
-C'\eps |Df|(\Om)
\end{align*}
by \eqref{eq:pV estimate}.
Letting $\eps\to 0$, we get the result.

In the general case we have $\Om=\bigcup_{k=1}^{\infty}\Om_k$, where the $\Om_k$'s
are disjoint open intervals, perhaps empty.
For $M\in\N$, we have 
\begin{align*}
\liminf_{\lambda \to \infty}F_{\gamma,\lambda}\left(f,\Om\right)
&\ge \sum_{k=1}^M\liminf_{\lambda \to \infty}F_{\gamma,\lambda}\left(f,\Om_k\right)\\
&\ge \frac{2}{\gamma+1}\sum_{k=1}^M|Df|(\Om_k).
\end{align*}
Letting $M\to\infty$, we get
\[
\liminf_{\lambda \to \infty}F_{\gamma,\lambda}\left(f,\Om\right)
\ge \frac{2}{\gamma+1}|Df|(\Om).
\]
\end{proof}

	\begin{proposition}\label{thm:1d case variation abs cont}
	Let $\gamma>0$, let  $\Om\subset \R$ be open and let 
	$f\in L^1_{\loc}(\Om)$ with $\Var(f,\Om)<\infty$.
	Then
	\[
	\liminf_{\lambda \to \infty}F_{\gamma,\lambda}\left(f,\Om\right)
	\ge \frac{2}{\gamma}|D^a f|(\Om).
	\]
\end{proposition}
The proof is essentially the same as one given in \cite{Pic}, but we provide it for completeness.

\begin{proof}
	Again we can assume $f$ to be the pointwise representative \eqref{eq:pw repr}.
	Then we have that $f$ is differentiable at a.e. $x\in \Om$, with
	\[
	\int_{\Om}|f'(x)|\,dx=|D^a f|(\Om);
	\]
	see \cite[Section 3.2]{AFP}.
	Fix $0<\eps<1$.
	For every $0<\lambda<\infty$, let
\begin{align*}
 A_{\lambda,\eps}:=\Bigg\{x\in \Om\colon \frac{|f(x)-f(y)|}{y-x}\ge (1-\eps)|f'(x)|
 \textrm{ for all }y\in \left(x,x+\left(\frac{(1-\eps)|f'(x)|}{\lambda}\right)^{1/\gamma}\right)\Bigg\}.
\end{align*}
If $x\in A_{\lambda,\eps}$, then for every 
\[
y\in \left(x,x+\left(\frac{(1-\eps)|f'(x)|}{\lambda}\right)^{1/\gamma}\right)
\]
we have
\[
|f(y)-f(x)|>(1-\eps)(y-x)|f'(x)|>\lambda (y-x)^{1+\gamma}.
\]
Thus
\begin{align*}
\nu_{\gamma}(E_{\gamma,\lambda}(f,\Om))
&\ge 2\int_{A_{\lambda,\eps}}\int_{x}
	^{x+\left(\frac{(1-\eps)|f'(x)|}{\lambda}\right)^{1/\gamma}}(y-x)^{\gamma-1}\,dy\,dx\\
	&=2\frac{1-\eps}{\gamma\lambda}\int_{A_{\lambda,\eps}}|f'(x)|\,dx.
\end{align*}
Note that a.e. $x\in \Om$ belongs to $A_{\lambda,\eps}$ for sufficiently large $\lambda>0$.
Thus we get
\begin{align*}
		\liminf_{\lambda \to \infty}F_{\gamma,\lambda}\left(f,\Om\right)
		& = 	\liminf_{\lambda \to \infty}\lambda \nu_{\gamma}(E_{\gamma,\lambda}(f,\Om))\\
		&\ge 2\frac{1-\eps}{\gamma}\liminf_{\lambda \to \infty}\int_{A_{\lambda,\eps}}|f'(x)|\,dx\\
		&= 2\frac{1-\eps}{\gamma}\int_{\Om}|f'(x)|\,dx\\
		&= 2\frac{1-\eps}{\gamma}|D^a f|(\Om).
\end{align*}
Letting $\eps\to 0$, we get the result.
\end{proof}
	
	The following is our main result in the 1-dimensional case.
	
		\begin{proposition}\label{thm:1d case variation abs cont and sing}
		Let $\gamma>0$, let  $\Om\subset \R$ be open and let 
		$f\in L^1_{\loc}(\Om)$ with $\Var(f,\Om)<\infty$.
		Then
		\[
		\liminf_{\lambda \to \infty}F_{\gamma,\lambda}\left(f,\Om\right)
		\ge \frac{2}{\gamma}|D^a f|(\Om)+\frac{2}{\gamma+1}|D^s f|(\Om).
		\]
	\end{proposition}
\begin{proof}
	Fix $\eps>0$.
We find an open set $W'\subset \Om$ such that $|D^a f|(W')<\eps$ and $|D^s f|(\Om\setminus W')<\eps$.
We also find an open set $W\Subset W'$ such that $|Df|(W\setminus W')<\eps$.
Now
\begin{align*}
	\liminf_{\lambda \to \infty}F_{\gamma,\lambda}\left(f,\Om\right)
	&\ge \liminf_{\lambda \to \infty}F_{\gamma,\lambda}\left(f,W\right)
	+\liminf_{\lambda \to \infty}F_{\gamma,\lambda}\left(f,\Om\setminus \overline{W}\right)\\
	&\ge \frac{2}{\gamma+1}|D f|(W)+\frac{2}{\gamma}|D^a f|(\Om\setminus \overline{W})
	\quad\textrm{by Propositions }\ref{thm:1d case variation},\ref{thm:1d case variation abs cont}\\
	&\ge \frac{2}{\gamma+1}\left(|D^s f|(\Om)-2\eps\right) 
	+ \frac{2}{\gamma}\left(|D^a f|(\Om)-\eps\right).
\end{align*}
Letting $\eps\to 0$, we get the result.
\end{proof}
	
	Now the generalization from the 1-dimensional to the $n$-dimensional case is standard.
	Denote the unit sphere in $\R^n$ by $\mathbb S^{n-1}$.
	For a unit vector $\sigma\in \mathbb S^{n-1}$, denote by $\langle \sigma\rangle^\perp$ the $(n-1)$-dimensional hyperplane perpendicular to $\sigma$ and containing the origin.
	For an open set $\Om\subset \R^n$ and $z\in\langle \sigma\rangle^\perp$, also denote
	$\Om_{\sigma,z}:=\{t\in \R\colon z+\sigma t\in \Om\}$ and
	\[
	f_{\sigma,z}(t):=f(z+\sigma t),\quad  t\in \Om_{\sigma,z}.
	\]
	
	\begin{proof}[Proof of Theorem \ref{thm:main}]
	We note the following facts that can be found in \cite[Section 3]{AFP}:
	for $\sigma\in \mathbb S^{n-1}$,
		\begin{align*}
			\int_{\Om}\,d|D^a f\cdot \sigma|
			&=\int_{\langle \sigma\rangle^\perp}|D^a f_{\sigma,z}|(\Om_{\sigma,z})\,dz,
		\end{align*}
	and the analogous fact with the absolutely continuous parts replaced by the singular parts,
		and then
		\begin{equation}\label{eq:change of var for Df}
			C_n\int_{\Om}\,d|D^a f|
			=\int_{\mathbb S^{n-1}}\int_{\langle \sigma\rangle^\perp}|D^a f_{\sigma,z}|(\Om_{\sigma,z})\,dz\,d\sigma,
		\end{equation}
		and again the analogous fact with the absolutely continuous parts replaced by the singular parts.
	Antonucci et al. \cite[p. 622]{AGMP} note that 
	for $g\in L^1(\Om\times\Om)$, we have the change of variables formula
	\[
	\int_{\R^n}\int_{\R^n} g(x',y')\,dx'\,dy'
	=\frac 12 \int_{\mathbb S^{n-1}} \int_{\langle \sigma\rangle^\perp}
	\int_{\R}\int_{\R}g(z+\sigma x,z+\sigma y)|x-y|^{n-1}\,dy\,dx\,dz\,d\sigma.
	\]
	Thus we have
	\begin{equation}\label{eq:Om multiple integral}
	\begin{split}
	\int_{\Om}\int_{\Om} g(x',y')\,dx'\,dy'
	 =\frac 12 \int_{\mathbb S^{n-1}} \int_{\langle \sigma\rangle^\perp}
	\int_{\Om_{\sigma,z}}\int_{\Om_{\sigma,z}}g(z+\sigma x,z+\sigma y)|x-y|^{n-1}\,dy\,dx\,dz\,d\sigma.
	\end{split}
	\end{equation}
	We get
	\begin{align*}
	&\liminf_{\lambda\to\infty}\lambda \nu_\gamma\left(\left\{(x',y')\in \Om\times \Om\colon
	\frac{|f(x')-f(y')|}{|x'-y'|^{1+\gamma}}>\lambda\right\}\right)\\
	& =\liminf_{\lambda\to\infty}\lambda\int_{\Om}\int_{\Om} \mathbbm{1}_{\{(x',y')\in \Om\times \Om\colon
		|f(x')-f(y')|/|x'-y'|^{1+\gamma}>\lambda\}}(x',y')|x'-y'|^{\gamma-n}\,dx'\,dy'\\
	&= \liminf_{\lambda\to\infty}\frac {\lambda}{2}\int_{\mathbb S^{n-1}}\int_{\langle \sigma\rangle^\perp}
		\quad \textrm{by }\eqref{eq:Om multiple integral}\\
	&\quad\int_{\Om_{\sigma,z}}\int_{\Om_{\sigma,z}}\mathbbm{1}_{\{(x,y)\in \Om_{\sigma,z}
		\times \Om_{\sigma,z}\colon
		|f(z+\sigma x)-f(z+\sigma y)|/|x-y|^{1+\gamma}>\lambda\}}(x,y)
	 \cdot |x-y|^{\gamma-1}\,dy\,dx\,dz\,d\sigma\\
	 &\ge \frac {1}{2}\int_{\mathbb S^{n-1}}\int_{\langle \sigma\rangle^\perp}\liminf_{\lambda\to\infty}\lambda\int_{\Om_{\sigma,z}}\int_{\Om_{\sigma,z}}
	 \quad \textrm{by Fatou's lemma}\\
	 &\quad\ \mathbbm{1}_{\{(x,y)\in \Om_{\sigma,z}
	 	\times \Om_{\sigma,z}\colon
	 	|f(z+\sigma x)-f(z+\sigma y)|/|x-y|^{1+\gamma}>\lambda\}}(x,y)
	 \cdot |x-y|^{\gamma-1}\,dy\,dx\,dz\,d\sigma\\
	& \ge \frac{1}{\gamma}\int_{\mathbb S^{n-1}}\int_{\langle \sigma\rangle^\perp}
	|D^a f_z|(\Om_{\sigma,z})\,dz\,d\sigma
	+\frac{1}{\gamma+1}\int_{\mathbb S^{n-1}}\int_{\langle \sigma\rangle^\perp}
	|D^s f_z|(\Om_{\sigma,z})\,dz\,d\sigma
	\quad\textrm{by Prop. }\ref{thm:1d case variation abs cont and sing}\\
	& = \frac{C_n}{\gamma}|D^a f|(\Om)+\frac{C_n}{\gamma+1}|D^s f|(\Om)
	\end{align*}
	by \eqref{eq:change of var for Df}.
	\end{proof}

In the following example, we show that the coefficient of $|D^cf|(\Om)$
in Theorem \ref{thm:main} is sharp.
The example is based on a standard Cantor-type construction considered previously for example
in \cite[Example 6.3]{DPFP}.

\begin{example}\label{ex:sharpness}
Construct a generalized Cantor-type function as follows.
Let $0<\alpha_j\le 1/2$ for each $j\in\N$;
in fact we will choose either $0<\alpha_j\le 1/4$ or $\alpha_j=1/2$ for each $j\in\N$.
Let $I^0_1:=[0,1]$.
Let $I_1^1:=[0,\alpha_1]$, $J_1^1:=(\alpha_1,1-\alpha_1)$, and $I_2^1:=[1-\alpha_1,1]$.
Recursively, divide each $I^j_k$, with $j\in\N$ and $k=1,\ldots,2^j$, into three subintervals,
which are denoted $I_{2k-1}^{j+1}$, $J_{k}^{j+1}$, and $I_{2k}^{j+1}$, such that the open interval
$J_{k}^{j+1}$ is centered at the center point of the closed interval $I^j_k$, with length $(1-2\alpha_{j+1})|I^j_k|$.

Define $f_j$ to be a continuous, increasing function in $[0,1]$, with $f(0)=0$ and $f(1)=1$, and such that
\[
f_j(x):=\frac{2k-1}{2^l}\quad\textrm{for }x\in J_k^l,\ \textrm{with }l\le j\textrm{ and }1\le k\le 2^{l-1},
\]
and $f_j$ is affine on each $I^j_k$ , with $k=1,\ldots,2^{j}$. 
If an interval $J^l_k$ is empty (as happens when $\alpha_l=1/2$),
then we also let $f_j$ be affine on the union of the two adjacent ``$I$-intervals''.
Then let
$f$ be the pointwise (and uniform) limit of the $f_j$'s. Consider $f$ in the domain $\Om:=(0,1)$.
It is clear that
$|Df|(\Om)=1$.

Let $\gamma>0$.
Since $\nu_{\gamma}$ is a finite measure on $\Om\times \Om$, absolutely continuous
with respect to the measure $\mathcal L^1\times\mathcal L^1$, it is easy to see that
for any functions $g,h$ on $\Om$ and a fixed $\widehat{\lambda}>0$, we have
\begin{equation}\label{eq:g and h functions}
	\forall\eps>0\textrm{ }\exists \delta>0\textrm{ such that if }
	\mathcal L^1(\{g\neq h\})<\delta,\textrm{ then }
	|F_{\gamma,\widehat{\lambda}}(g,\Om)-F_{\gamma,\widehat{\lambda}}(h,\Om)|<\eps.
\end{equation}

Recursively, for each $j=1,2,\ldots$ we choose $\lambda_j$ and then $\alpha_j$ as follows.
Consider a fixed $j\in\N$.
First choose a number
\begin{equation}\label{eq:lambda j}
\lambda_j\ge 2^{-j}\times 2^{\gamma+1}|I^{j-1}_1|^{-\gamma-1}.
\end{equation}
Now if $(x,y)\in E_{\gamma,\lambda_j}(f,\Om)$ and $\alpha_j<1/4$, then necessarily
$x,y$ belong to the convex hull $\mathrm{co}(J\cup J')$
of a nonempty ``$J$-interval'' $J$ and the next nonempty ``$J$-interval'' to the right, denoted by $J'$. 
(These ``$J$-intervals'' can be different sizes; note that we obtain empty ``$J$-intervals''
when $\alpha_k=1/2$ for some $k$.)
If we let $\alpha_j\searrow 0$, then $f$ on $\mathrm{co}(J\cup J')$
converges pointwise to a function that takes only two values, with a jump of size
$s\ge (1/2)^j$ at the middle point.
Thus we have
\begin{align*}
&\limsup_{\alpha_j\to 0}\lambda_j\nu_{\gamma} \left( \{(x,y)\in (\mathrm{co}(J\cup J'))^2 \colon|f(x)-f(y)|>\lambda_j|x-y|^{1+\gamma}\}\right)\\
&\qquad \le s\frac{2}{1+\gamma};
\end{align*}
recall \eqref{eq:one step function} and note that
the limit is uniform for all choices of the values of $\alpha_{j+1},\alpha_{j+2},\ldots$,
due to \eqref{eq:g and h functions}.
In total, taking the sum below over all nonempty ``$J$-intervals'', we get
\begin{equation}\label{eq:F gamma lambda j}
\begin{split}
F_{\gamma,\lambda_j}\left(f,\Om\right)
&\le \lambda_j \sum_{J}\nu_{\gamma}\left(\left\{(x,y)\in (\mathrm{co}(J\cup J'))^2 \colon|f(x)-f(y)|>\lambda_j|x-y|^{1+\gamma}\right\}\right)\\
&\le(1+1/j) \frac{2}{1+\gamma}
\end{split}
\end{equation}
if we choose $\alpha_j$ sufficiently small.
Now we have from \eqref{eq:SBV limit}, noting that it clearly holds also with $\R$ replaced 
by the unit interval $\Om$, that
\[
\lim_{\lambda\to\infty}F_{\gamma,\lambda}(f_j,\Om)=\frac{2}{\gamma}.
\]
Thus we can choose a sufficiently large $\lambda_j'>\lambda_j+1$ such that
\begin{equation}\label{eq:lambda prime choice}
F_{\gamma,\lambda_j'}(f_j,\Om) > \frac{2}{\gamma}-\frac{1}{j}.
\end{equation}
Furthermore, we can perform the choice of $\lambda_j'$ in such a way that 
\[
\nu_{\gamma}(	\left\{(x,y)\in \Om\times \Om\colon
|f_j(x)-f_j(y)|=\lambda_j'|x-y|^{1+\gamma}\right\})=0,
\]
since this is true for all except at most countably many choices of $\lambda_j'$.
Next choose $\alpha_{j+1},\ldots,\alpha_{k}$ to take the value $1/2$, and 
$\lambda_{j+1},\ldots,\lambda_{k}$ and $\lambda'_{j+1},\ldots,\lambda'_{k}$ to all take the
value $\lambda_j'$, for $k>j$ sufficiently large that
\begin{equation}\label{eq:lambda prime}
F_{\gamma,\lambda_j'}(f,\Om) > \frac{2}{\gamma}-\frac{1}{j}
\end{equation}
regardless of the choices of $\alpha_{k+1},\alpha_{k+2},\ldots$;
this is possible by \eqref{eq:lambda prime choice} and Lemma \ref{lem:uniform convergence}.
Then choose 
\[
\lambda_{k+1}\ge \max\{\lambda_j',2^{-(k+1)+\gamma+1}|I^{(k+1)-1}_1|^{-\gamma-1}\},
\]
and we repeat the process starting from \eqref{eq:lambda j}.
Note that the $\lambda_j $'s are an increasing sequence with limit $+\infty$,
and the same is true of the $\lambda'_j $'s.
Thus we get
\[
\liminf_{\lambda \to \infty}F_{\gamma,\lambda}(f,\Om)
\le \liminf_{j\to  \infty}F_{\gamma,\lambda_j}(f,\Om)
\le \frac{2}{1+\gamma}
\]
by \eqref{eq:F gamma lambda j},
and so the coefficient of $|D^cf|(\Om)$
in Theorem \ref{thm:main} is sharp.
We also get
\[
\limsup_{\lambda \to \infty}F_{\gamma,\lambda}\left(f,\Om\right)
\ge \limsup_{j\to  \infty}F_{\gamma,\lambda_j'}\left(f,\Om\right)
\ge \frac{2}{\gamma}
\]
by \eqref{eq:lambda prime}, showing that the limit does not exist.
\end{example}

\begin{remark}
If Theorem \ref{thm:1d case preli} turns out to be true also without the assumption of $f$ being
increasing, which seems likely, this could be used to prove that the $\Gamma$-limit of
the functionals $F_{\gamma,\lambda}(f,\Om)$ as $\lambda\to \infty$ is 
\[
\frac{C_n}{\gamma+1}|Df|(\Om).
\]
\end{remark}

		\noindent Address:\\
		
		\noindent Academy of Mathematics and Systems Science\\
		Chinese Academy of Sciences\\
		Beijing 100190, PR China\\
		E-mail: {\tt panulahti@amss.ac.cn}
	
\end{document}